\documentclass[11pt]{article}

\usepackage{amsmath,amssymb,amsthm,mathtools}

\usepackage[a4paper,margin=1in]{geometry}

\usepackage{hyperref}

\newtheorem{theorem}{Theorem}
\newtheorem{lemma}[theorem]{Lemma}

\newtheorem{corollary}[theorem]{Corollary}

\theoremstyle{definition}

\title{Strong uniqueness for fractional stochastic differential equations  with sign drift}
\author{ Ercan S\"onmez\thanks{Department of Mathematics, Pohang University of Science and Technology, \texttt{soenmez@postech.ac.kr}} \and Yuliya Mishura\thanks{Department of Probability, Statistics and Actuarial Mathematics, Taras Shevchenko National University of Kyiv, \texttt{yuliyamishura@knu.ua}}} 
\date{\today}

\begin{document}
	
\maketitle

\begin{abstract}
	Establishing pathwise uniqueness for stochastic differential equations driven by fractional Brownian motion with Hurst parameter \(H>1/2\) remains challenging, particularly for discontinuous drifts beyond the regularity regimes of existing theories. A canonical example is the sign drift, for which uniqueness fails in the corresponding deterministic equation. In this paper, we complement recent strong existence results by establishing pathwise uniqueness for the equation with sign drift perturbed by fractional Brownian motion for every \(H\in(1/2,1)\), thus obtaining strong well-posedness. Our proof develops a new approach based on the infinite iteration of an inequality arising from a fixed-point relation, combined with tail probability estimates exploiting the Gaussian structure of the noise. The methods and results cover a broader class of Gaussian Volterra processes satisfying suitable local nondeterminism and positivity conditions.
\end{abstract}



\section{Introduction}

We consider the stochastic differential equation
\[
X_t=x+\int_0^t\operatorname{sign}(X_s)\,ds+B_t^H,\qquad t\geq 0,
\]
where \(B^H\) is a fractional Brownian motion with Hurst parameter \(H\in(1/2,1)\). The sign function provides one of the most canonical examples illustrating the phenomenon of regularization by noise: for the corresponding deterministic equation, the uniqueness of the solution fails,  yet the addition of a suitable stochastic perturbation can restore uniqueness.

This phenomenon was first established for stochastic differential equations driven by standard Brownian motion in the seminal work of Zvonkin \cite{Zvonkin1974}. For fractional Brownian motion, Nualart and Ouknine \cite{NualartOuknine2002} established strong well-posedness for the equation with sign drift in the regime \(H\in(0,1/2)\).  In fact, their pioneering work treated a much broader class of drift coefficients and revealed a fundamental distinction between the two Hurst regimes. For \(H\le 1/2\) a linear growth condition on the drift was sufficient; the case \(H>1/2\) required Hölder continuity in the spatial variable of order \(\alpha>1-{1}/{(2H)}\). Thus, their results exhibit a transition in the regularity requirements at \(H=1/2\). These results were subsequently developed and refined, notably through the regularization-by-noise framework of Catellier and Gubinelli~\cite{CatellierGubinelli2016}. Since then, a substantial literature has emerged on stochastic differential equations with irregular drifts driven by fractional Brownian motion; see, for instance, \cite{AnzelettiRichardTanre2023, BanosNilssenProske2020, butkovsky0, butkovsky1, butkovsky2, GaleatiHarangMayorcas2023, le2020stochastic}. 

The framework introduced by Catellier and Gubinelli has been particularly influential in the subsequent development of regularization-by-noise theory. One of its central results established pathwise uniqueness for drifts belonging to suitable Hölder–Besov spaces \(\mathcal C^\alpha\), under the regularity condition
\[
\alpha>1-\frac{1}{2H}.
\]

For \(H>1/2\), the right-hand side is strictly positive, so neither this condition nor the regularity assumptions imposed in subsequent strong well-posedness results cover discontinuous drifts such as the sign function. Beyond this regularity regime, pathwise uniqueness remains poorly understood, particularly for fractional Brownian motion with \(H>1/2\). Establishing pathwise uniqueness and strong well-posedness for equations with irregular drifts remains a challenging problem in general.

Strong well-posedness for the equation with sign drift in the regime \(H>1/2\) has long remained an open question. This equation has received particular attention in the works of Mishura and Nualart \cite{mishuranualart} and Boufoussi and Ouknine \cite{boufoussi}; see also \cite{mishura2008stochastic}. Their approaches, which relied on a right-continuous convention for the sign function, established existence results but did not yield pathwise uniqueness for the repulsive sign drift, leaving the question of strong well-posedness unresolved.

In this paper, we establish pathwise uniqueness for the equation above for every \(H\in(1/2,1)\). Combined with recent strong existence results for this equation \cite{ButhenhoffSonmez2025}, this yields strong well-posedness throughout the regime \(H>1/2\).

Our proof introduces a new approach based on the infinite iteration of an inequality derived from a fixed-point relation, combined with tail probability estimates exploiting the Gaussian structure of the driving noise. This approach differs from existing regularization-by-noise techniques and may be of independent interest. Moreover, our methods and results extend to a broader class of Gaussian Volterra processes satisfying suitable local nondeterminism and positivity conditions.

The remainder of the paper is organized as follows. In Section \ref{sec:main}, we formulate our main results and explain the ideas behind the methods developed in this paper. Section \ref{sec:proof} is devoted to the proofs.

\section{Strong uniqueness of the solution}\label{sec:main}

We begin by introducing the mathematical framework. Let \(T\in(0,\infty)\) and consider a filtered probability space \((\Omega,(\mathcal{A}_t)_{t\in[0,T]},\mathbb{P})\). Let \(W=(W_t)_{t\in[0,T]}\) be a standard Brownian motion, and assume that \((\mathcal{A}_t)_{t\in[0,T]}\) is the usual augmentation of the natural filtration generated by \(W\).

Following \cite{ButhenhoffSonmez2025}, we consider stochastic differential equations driven by Gaussian Volterra processes of the form
\[
B_t^K=\int_0^t K(t,s)\,dW_s,\qquad t\in[0,T],
\]
where \(K\) is a square-integrable Volterra kernel. We assume that \(B^K\) admits continuous sample paths.

We impose the following local nondeterminism condition on the kernel: there exist constants \(H\in(0,1)\) and \(c>0\) such that
\[
\int_{t-\varepsilon}^t K(t,s)^2\,ds
\ge c\varepsilon^{2H},
\qquad 0<\varepsilon<t\le T.
\]

In addition, we assume the following positivity condition on the kernel:
\[
K(t,r)K(s,r)\ge0
\quad\text{for almost every }r\in(0,s),
\qquad 0<s\le t\le T.
\]
This condition is automatically satisfied whenever \(K\) is nonnegative. 

We emphasize that the parameter \(H\) in the local nondeterminism condition may take any value in \((0,1)\). In particular, both conditions are satisfied by the standard Volterra kernels of fractional Brownian motion. Thus, our approach also recovers the known pathwise uniqueness for fractional Brownian motion with \(H\le 1/2\) by a new method. Let us emphasize that the main novelty, however, lies in the regime \(H>1/2\), which is the primary focus and motivation of this work. Other examples include Riemann–Liouville processes for every \(H\in(0,1)\), as well as nonnegative mixtures of completely correlated fractional Brownian motions, see \cite{ButhenhoffSonmez2025}.

The following statement is the main result of this paper.

\begin{theorem}\label{thm:sign-pathwise-uniqueness}
	Let \(B^K=(B_t^K)_{t\in[0,T]}\) be a continuous Gaussian Volterra process whose kernel \(K\) satisfies the local nondeterminism and positivity conditions stated above, and let \(x\in\mathbb R\). Consider the stochastic differential equation 
    	\begin{equation}\label{eq:sign-equation}
    		X_t
    		=
    		x+B_t^K+\int_0^t \operatorname{sign}(X_s)\,ds,
    		\qquad t\in [0,T],
    	\end{equation}
    	where \(\operatorname{sign}(0)=0\).
	
	Then pathwise uniqueness holds for \eqref{eq:sign-equation}.
	More precisely, if \(X^1\) and \(X^2\) are two continuous
	solutions driven by the same Gaussian Volterra process and
	having the same initial value \(x\), then
	\[
	\mathbb P\left(
	X_t^1=X_t^2\text{ for all }t\in [0,T]
	\right)=1.
	\]
\end{theorem}

Combining our pathwise uniqueness result with the strong existence result established in~\cite{ButhenhoffSonmez2025}, we obtain the following strong well-posedness result.

\begin{corollary}
Let \(B^K=(B_t^K)_{t\in[0,T]}\) be a continuous Gaussian
Volterra process whose kernel \(K\) satisfies the local
nondeterminism and positivity conditions stated above.
Then, for every \(x\in\mathbb R\), strong well-posedness holds for equation \eqref{eq:sign-equation}.
\end{corollary}

Before we prove this result, let us explain the idea  behind  the methods used here.

\subsection{Idea behind the proof}
The main idea of the proof is to exploit the fact that  the drift part of any solution can be determined iteratively from the path of the Gaussian Volterra process. To see this, write a solution in the form
\[
X_t=x+B_t^K+A_t,
\]
where
\[
A_t:=\int_0^t \operatorname{sign}(X_s)\,ds.
\]
Substituting the representation of \(X\) into the definition of \(A\) gives the fixed-point relation
\[
A_t
=
\int_0^t
\operatorname{sign}(x+B_s^K+A_s)\,ds.
\]
Thus, if we define
\[
\Phi(A)_t
:=
\int_0^t
\operatorname{sign}(x+B_s^K+A_s)\,ds,
\]
then the possible solutions of the original equation correspond precisely to fixed points of the map \(\Phi\).

This formulation suggests an iterative procedure. Before using the equation at all, we only know from \(|\operatorname{sign}|\le1\) that
\[
-t\le A_t\le t.
\]
This is our zeroth-order information. We may now insert these two extreme possibilities into the right-hand side of the fixed-point equation. Since \(\operatorname{sign}\) is nondecreasing, this gives the improved bounds
\[
\Phi(L_0)_t
\le A_t\le
\Phi(U_0)_t,
\]
where
\[
L_0(t)=-t,\qquad U_0(t)=t.
\] 
The new bounds can in turn be inserted into the fixed-point equation, producing still sharper bounds. Repeating this procedure leads to the monotone sequences
\[
L_0(t)=-t,\qquad U_0(t)=t,
\]
and
\[
L_{n+1}=\Phi(L_n),
\qquad
U_{n+1}=\Phi(U_n).
\]
Every possible solution satisfies
\[
L_n(t)\le A_t\le U_n(t)
\qquad\text{for every }n.
\]

The uniqueness problem is therefore reduced to understanding whether the uncertainty interval
\[
[L_n(t),U_n(t)]
\]
collapses as $n\to\infty$. Denote
\[
M_n(t)=\frac{U_n(t)+L_n(t)}2,
\qquad
R_n(t)=\frac{U_n(t)-L_n(t)}2.
\]
The quantity \(R_n(t)\) measures precisely the uncertainty remaining after \(n\) iterations. Hence it is enough to prove
\[
R_n(t)\longrightarrow0\quad \text{a.s. for any}\quad t>0 \quad\text{as}\quad n\to\infty.
\]

The key observation behind the proof is that the uncertainty can
survive from one iteration to the next only when the midpoint of
the current uncertainty interval is sufficiently close to zero.
To see this, write
\[
L_n(t)=M_n(t)-R_n(t),
\qquad
U_n(t)=M_n(t)+R_n(t).
\]
Thus \(M_n(t)\) is the midpoint of the interval in which the
drift part \(A_t\) is known to lie after \(n\) iterations.

Applying the fixed-point map once more, the two extremal
arguments of the sign function at time \(s\) are
\[
x+B_s^H+M_n(s)-R_n(s)
\qquad\text{and}\qquad
x+B_s^H+M_n(s)+R_n(s).
\]
Introduce
\[
Y_s^{(n)}
:=
x+B_s^K+M_n(s).
\]
If
\[
Y_s^{(n)}>R_n(s),
\]
then both extremal arguments are positive, and hence both sign
functions equal \(+1\). Similarly, if
\[
Y_s^{(n)}<-R_n(s),
\]
then both arguments are negative and both sign functions equal
\(-1\). In either case, the lower and upper iterations receive
exactly the same contribution at time \(s\), so no uncertainty
survives at that time.

Consequently, the integrands defining the upper and lower iterations
coincide whenever
\[
    |Y_s^{(n)}|>R_n(s).
\]
Thus, their difference can be nonzero only on the set
\[
    \bigl\{s\in[0,t]:|Y_s^{(n)}|\le R_n(s)\bigr\}.
\]
It follows that
\begin{align*}
    R_{n+1}(t)
    &=
    \frac12\int_0^t
    \Bigl[
        \operatorname{sign}\bigl(Y_s^{(n)}+R_n(s)\bigr)
        -
        \operatorname{sign}\bigl(Y_s^{(n)}-R_n(s)\bigr)
    \Bigr]\,ds\\
    &\le
    \int_0^t
    \mathbf 1_{\{|Y_s^{(n)}|\le R_n(s)\}}\,ds.
\end{align*}


This inequality is recursive. The uncertainty remaining after step \(n+1\) is controlled by
the amount of time the midpoint process \(Y^{(n)}\) spends in
the ball of radius \(R_n\) around zero. Thus, once \(R_n\)
is small, uncertainty can survive only during visits to an
increasingly small neighbourhood of zero. The probabilistic part
of the proof quantifies how little time can be spent in these
small balls. This produces a further improvement of \(R_n\), and
iterating this improvement ultimately forces
\[
R_n(t)\longrightarrow0.
\]
Therefore the lower and upper bounds collapse onto the same path. Since every solution is
trapped between them, this yields pathwise uniqueness.

\section{Proof of Theorem \ref{thm:sign-pathwise-uniqueness}}\label{sec:proof}

We begin this section with some notation. 

Put
	\[
	a:=1-H\in(0,1).
	\]
	For a continuous function \(A\), define
	\[
	\Phi(A)_t
	:=
	\int_0^t
	\operatorname{sign}(x+B_s^K+A_s)\,ds.
	\]
	Then \(X\) solves \eqref{eq:sign-equation} if and only if
	\[
	X_t=x+B_t^K+A_t
	\]
	for some fixed point
	\begin{equation*}
		A=\Phi(A).
	\end{equation*}
	Every such fixed point satisfies
	\begin{equation}\label{eq:A-speed-bound}
		|A_t|\le t,\qquad t\ge0.
	\end{equation}
	We shall prove that \(\Phi\) has at most one fixed point. Moreover, for a deterministic continuous path \(w\), write
	\[
	\Phi_w(f)_t
	:=
	\int_0^t
	\operatorname{sign}(x+w_s+f_s)\,ds.
	\]
	Since the sign function is nondecreasing, the following implication holds:
	\begin{equation}\label{eq:Phi-monotone}
		w\le\widetilde w,\quad f\le\widetilde f
		\quad\Longrightarrow\quad
		\Phi_w(f)\le\Phi_{\widetilde w}(\widetilde f).
	\end{equation}
	For the random path \(w=B^K\), define
	\[
	L_0(t):=-t,
	\qquad
	U_0(t):=t,
	\]
	and recursively
	\begin{equation}\label{eq:extremal-iteration}
		L_{n+1}:=\Phi(L_n),
		\qquad
		U_{n+1}:=\Phi(U_n).
	\end{equation}
	Since every image under \(\Phi\) satisfies
	\[
	-t\le\Phi(f)_t\le t,
	\]
	we have
	\[
	L_0\le L_1,
	\qquad
	U_1\le U_0.
	\]
	Using \eqref{eq:Phi-monotone} inductively gives
	\begin{equation*}
		L_0\le L_1\le\cdots\le L_n
		\le U_n\le\cdots\le U_1\le U_0.
	\end{equation*}
	If \(A\) is any fixed point of \(\Phi\), then
	\eqref{eq:A-speed-bound} implies
	\[
	L_0\le A\le U_0.
	\]
	Applying \(\Phi\) repeatedly and using monotonicity yields
	\begin{equation}\label{eq:fixed-point-squeeze}
		L_n\le A\le U_n,
		\qquad n\ge0.
	\end{equation}
	Now define
	\begin{equation*}
		M_n(t):=\frac{U_n(t)+L_n(t)}{2},
		\qquad
		R_n(t):=\frac{U_n(t)-L_n(t)}{2},
	\end{equation*}
	and set
	\[
	Y_t^{(n)}
	:=
	x+B_t^K+M_n(t).
	\]
	The following is the first auxiliary result of this section.
	
	\begin{lemma}
		It holds
		\begin{equation}\label{eq:R-recursion}
			R_{n+1}(t)
			\le
			\int_0^t
			\mathbf 1_{\{
				|Y_s^{(n)}|\le R_n(s)
				\}}\,ds.
		\end{equation}
	\end{lemma}
	
	\begin{proof}
		By definition
			\[
		R_n(t)\ge0,\qquad R_0(t)=t,
		\]
		and every fixed point satisfies
		\begin{equation*}
			|A_t-M_n(t)|\le R_n(t).
		\end{equation*}
		By \eqref{eq:extremal-iteration},
		\begin{align*}
			R_{n+1}(t)
			&=
			\frac12
			\int_0^t
			\Big[
			\operatorname{sign}\bigl(
			Y_s^{(n)}+R_n(s)
			\bigr)
			-
			\operatorname{sign}\bigl(
			Y_s^{(n)}-R_n(s)
			\bigr)
			\Big]\,ds.
			\label{eq:R-exact}
		\end{align*}
		For every \(y\in\mathbb R\) and \(r\ge0\),
		\[
		0
		\le
		\frac12
		\left[
		\operatorname{sign}(y+r)
		-
		\operatorname{sign}(y-r)
		\right]
		\le
		\mathbf 1_{\{|y|\le r\}}.
		\]
		The inequality also holds when \(y=\pm r\), because
		\(\operatorname{sign}(0)=0\). Hence
		\begin{equation*}
			R_{n+1}(t)
			\le
			\int_0^t
			\mathbf 1_{\{
				|Y_s^{(n)}|\le R_n(s)
				\}}\,ds.
		\end{equation*}
	\end{proof}
	
	In the following, we work with the Volterra representation
    \[
    B_t^K=\int_0^t K(t,r)\,dW_r,
    \]
    where \(W\) is a standard Brownian motion and \(K\) satisfies the local nondeterminism and positivity conditions introduced above. The local nondeterminism condition provides the lower bound
    \[
    \int_{s}^{t}K(t,r)^2\,dr
    \ge c(t-s)^{2H},\qquad 0\le s<t\le T.
    \]
    The positivity condition ensures that
    \[
    \int_a^{u\wedge t}K(u,r)K(t,r)\,dr\ge0,
    \qquad 0\le a<t,\quad a\le u\le T.
    \]
    These two properties will be used in the following lemma.
	
	\begin{lemma}
		There exists a constant \(C_H<\infty\) such that
		for every \(n\ge0\), every integer \(p\ge1\), every
		\[
		0=s_0<s_1<\cdots<s_p,
		\]
		and every \(r_1,\ldots,r_p>0\),
		\begin{equation}\label{eq:multi-time-small-ball}
			\mathbb P\left(
			|Y_{s_i}^{(n)}|\le r_i
			\text{ for }i=1,\ldots,p
			\right)
			\le
			C_H^p
			\prod_{i=1}^p
			\frac{r_i}{(s_i-s_{i-1})^H}.
		\end{equation}
	\end{lemma}
	
	\begin{proof}
		
		For \(i=1,\ldots,p\), define
		\[
		\sigma_i^2
		:=
		\int_{s_{i-1}}^{s_i}
		K(s_i,r)^2\,dr
		\]
		and
		\[
		g_i(r)
		:=
		\frac{
			K(s_i,r)
			\mathbf 1_{(s_{i-1},s_i]}(r)
		}{\sigma_i}.
		\]
		The functions \(g_1,\ldots,g_p\) are orthonormal in
		\(L^2([0,s_p])\), since they have disjoint supports and unit
		\(L^2\)-norm. Hence
		\[
		Z_i:=\int_0^{s_p}g_i(r)\,dW_r,
		\qquad i=1,\ldots,p,
		\]
		are independent standard normal random variables.
		
		Decomposing the underlying isonormal Gaussian process into
		\(\operatorname{span}\{g_1,\ldots,g_p\}\) and its orthogonal
		complement gives the representation
		\begin{equation*}\label{eq:Gaussian-decomposition}
			B_u^K
			=
			\widetilde B_u
			+
			\sum_{i=1}^p Z_i h_i(u),
		\end{equation*}
		where
		\[
		h_i(u)
		:=
		\int_0^u K(u,r)g_i(r)\,dr.
		\]
		By positivity of \(K\),
		\begin{equation*}\label{eq:hi-positive}
			h_i(u)\ge0.
		\end{equation*}
		Furthermore, because \(g_i\) is supported on
		\((s_{i-1},s_i]\),
		\begin{equation*}\label{eq:hi-zero}
			h_i(u)=0,
			\qquad u\le s_{i-1}.
		\end{equation*}
		At \(u=s_i\),
		\begin{equation*}\label{eq:hi-sigma}
			h_i(s_i)
			=
			\frac1{\sigma_i}
			\int_{s_{i-1}}^{s_i}
			K(s_i,r)^2\,dr
			=
			\sigma_i.
		\end{equation*}
		
		If \(w\le\widetilde w\), then
		\eqref{eq:Phi-monotone} and induction imply
		\[
		L_n^w\le L_n^{\widetilde w},
		\qquad
		U_n^w\le U_n^{\widetilde w},
		\]
		and consequently
		\begin{equation*}\label{eq:M-monotone}
			M_n^w\le M_n^{\widetilde w},
		\end{equation*}
		where \(M_n^{w}\) denotes the midpoint of the lower and upper
		iterations constructed with driving path \(w\). For brevity, set
			\[
			A_i:=\bigl\{|Y_{s_i}^{(n)}|\le r_i\bigr\},
			\qquad i=1,\ldots,p.
			\]
			Define
				\begin{equation*}
				\mathcal G_i
				:=
				\sigma\left(
				\widetilde B_u:\,u\ge0;\,
				Z_j:\,j\neq i
				\right).
				\label{eq:def-Gi}
			\end{equation*}
			Since \(Z_i\) is independent of \(\mathcal G_i\), its conditional
			distribution given \(\mathcal G_i\) is again standard normal.
			
			The events \(A_1,\ldots,A_{i-1}\) are
			\(\mathcal G_i\)-measurable, see below for the proof. In particular, for \(i=p\), the tower
			property gives
			\begin{align}
				\mathbb P\left(\bigcap_{i=1}^p A_i\right)
				&=
				\mathbb E\left[
				\prod_{i=1}^p \mathbf 1_{A_i}
				\right]
				\notag\\
				&=
				\mathbb E\left[
				\mathbb E\left[
				\prod_{i=1}^p \mathbf 1_{A_i}
				\,\middle|\,
				\mathcal G_p
				\right]
				\right]
				\notag\\
				&=
				\mathbb E\left[
				\left(
				\prod_{i=1}^{p-1}\mathbf 1_{A_i}
				\right)
				\mathbb E\left[
				\mathbf 1_{A_p}
				\,\middle|\,
				\mathcal G_p
				\right]
				\right].
				\label{eq:first-conditioning-step}
			\end{align}
			We shall prove below that, for every \(i=1,\ldots,p\),
			\begin{equation}
				\mathbb P\left(
				A_i\,\middle|\,\mathcal G_i
				\right)
				=
				\mathbb E\left[
				\mathbf 1_{A_i}\,\middle|\,\mathcal G_i
				\right]
				\le
				\sqrt{\frac{2}{\pi}}\,
				\frac{r_i}{\sigma_i}
				\qquad\text{a.s.}
				\label{eq:conditional-bound}
			\end{equation}
			Assuming \eqref{eq:conditional-bound} for the moment,
			\eqref{eq:first-conditioning-step} yields
			\[
			\mathbb P\left(\bigcap_{i=1}^p A_i\right)
			\le
			\sqrt{\frac{2}{\pi}}\frac{r_p}{\sigma_p}
			\mathbb P\left(\bigcap_{i=1}^{p-1}A_i\right).
			\]
			Repeating the same argument successively with
			\(Z_{p-1},Z_{p-2},\ldots,Z_1\), we obtain
			\begin{equation}
				\mathbb P\left(\bigcap_{i=1}^p A_i\right)
				\le
				\left(\sqrt{\frac{2}{\pi}}\right)^p
				\prod_{i=1}^p\frac{r_i}{\sigma_i}.
				\label{eq:product-sigma}
			\end{equation}
			It therefore remains to prove the conditional estimate
			\eqref{eq:conditional-bound}. Fix \(i\in\{1,\ldots,p\}\). Recall the orthogonal decomposition
			\[
			B_u^K
			=
			\widetilde B_u
			+
			\sum_{j=1}^p Z_j h_j(u),
			\]
			where \(Z_1,\ldots,Z_p\) are independent standard normal random
			variables, the Gaussian process \(\widetilde B\) is independent of
			\((Z_1,\ldots,Z_p)\), and
			\[
			h_i(u)\ge0,
			\qquad
			h_i(u)=0 \quad\text{for }u\le s_{i-1},
			\qquad
			h_i(s_i)=\sigma_i.
			\]
	
			Conditionally on \(\mathcal G_i\), all the randomness in the driving
			path is therefore due to the single coordinate \(Z_i\).
			For \(z\in\mathbb R\), define the corresponding driving trajectory
			\[
			w^z(u)
			:=
			\widetilde B_u
			+
			\sum_{j\neq i} Z_jh_j(u)
			+
			zh_i(u),
			\]
			and set
			\begin{equation*}
				F_i(z)
				:=
				x+w^z(s_i)+M_n^{w^z}(s_i).
				\label{eq:def-Fi}
			\end{equation*}
			Then, conditionally
			on \(\mathcal G_i\),
			\[
			Y_{s_i}^{(n)}=F_i(Z_i).
			\]
			We stress that \(F_i\) need not be linear and
			\(Y_{s_i}^{(n)}\) need not be Gaussian. The only property of \(F_i\)
			that we use is the following deterministic estimate.
			
			Indeed, let \(z'>z\). Since \(h_i(u)\ge0\) for every \(u\),
			\[
			w^{z'}(u)-w^z(u)
			=
			(z'-z)h_i(u)
			\ge0,
			\qquad u\ge0.
			\]
			Hence \(w^z\le w^{z'}\) pointwise. By the monotonicity of the
			midpoint functional with respect to the driving path,
			\[
			M_n^{w^z}(s_i)
			\le
			M_n^{w^{z'}}(s_i).
			\]
			Using also \(h_i(s_i)=\sigma_i\), we obtain
			\begin{align}
				F_i(z')-F_i(z)
				&=
				w^{z'}(s_i)-w^z(s_i)
				+
				M_n^{w^{z'}}(s_i)-M_n^{w^z}(s_i)
				\notag\\
				&\ge
				(z'-z)h_i(s_i)
				=
				\sigma_i(z'-z).
				\label{eq:Fi-slope}
			\end{align}
			
			Consider now the set
			\[
			I_i
			:=
			\{z\in\mathbb R:|F_i(z)|\le r_i\}.
			\]
			If \(z<z'\) both belong to \(I_i\), then
			\[
			F_i(z),F_i(z')\in[-r_i,r_i],
			\]
			and hence
			\[
			F_i(z')-F_i(z)\le2r_i.
			\]
			Together with \eqref{eq:Fi-slope}, this implies
			\[
			z'-z\le\frac{2r_i}{\sigma_i}.
			\]
			Since \(F_i\) is increasing, \(I_i\) is an interval, possibly empty,
			and therefore
			\begin{equation}
				\operatorname{Leb}(I_i)
				\le
				\frac{2r_i}{\sigma_i},
				\label{eq:preimage-length}
			\end{equation}
			where $\operatorname{Leb}(I_i)$ denotes the Lebesgue measure of $I_i$.
            
			We may now estimate the conditional probability in
			\eqref{eq:conditional-bound}. Since, conditionally on
			\(\mathcal G_i\), \(Z_i\) has the standard normal density
			\[
			\varphi(z)
			=
			\frac{1}{\sqrt{2\pi}}e^{-z^2/2}
			\le
			\frac{1}{\sqrt{2\pi}},
			\]
			we have, almost surely,
			\begin{align*}
				\mathbb P\left(
				A_i\,\middle|\,\mathcal G_i
				\right)
				&=
				\mathbb P\left(
				|F_i(Z_i)|\le r_i
				\,\middle|\,
				\mathcal G_i
				\right)
				\notag\\
				&=
				\int_{\mathbb R}
				\mathbf 1_{\{|F_i(z)|\le r_i\}}
				\varphi(z)\,dz
				\notag\\
				&=
				\int_{I_i}\varphi(z)\,dz
				\notag\\
				&\le
				\frac{1}{\sqrt{2\pi}}
				\operatorname{Leb}(I_i)
				\notag\\
				&\le
				\sqrt{\frac{2}{\pi}}\,
				\frac{r_i}{\sigma_i},
				\label{eq:conditional-bound-proof}
			\end{align*}
			where the last inequality follows from
			\eqref{eq:preimage-length}. This proves
			\eqref{eq:conditional-bound}.
			
			It remains only to justify that $A_1,\ldots,A_{i-1}\in\mathcal G_i$, which is used in the
			iteration. If \(j<i\), then \(s_j\le s_{i-1}\), and therefore
			\[
			h_i(u)=0,
			\qquad u\le s_j.
			\]
			Consequently, changing \(Z_i\) leaves the driving path unchanged on
			\([0,s_j]\). Since \(M_n(s_j)\), and hence
			\(Y_{s_j}^{(n)}\), depends only on the driving path up to time
			\(s_j\), the event \(A_j\) does not depend on \(Z_i\). Thus
			\[
			A_1,\ldots,A_{i-1}\in\mathcal G_i,
			\]
			which justifies the successive conditioning from
			\(i=p\) down to \(i=1\).
			
			Combining \eqref{eq:product-sigma} with the kernel estimate
			\[
			\sigma_i^2
			\ge
			c_H(s_i-s_{i-1})^{2H},
			\]
			and hence
			\[
			\sigma_i
			\ge
			c_H(s_i-s_{i-1})^H
			\]
			after adjusting the constant, we finally obtain
			\[
			\mathbb P\left(
			\bigcap_{i=1}^p
			\{|Y_{s_i}^{(n)}|\le r_i\}
			\right)
			\le
			C_H^p
			\prod_{i=1}^p
			\frac{r_i}{(s_i-s_{i-1})^H},
			\]
			which is the desired estimate.
	\end{proof}

	For \(\alpha\ge1\), \(L>0\), and \(t\le T\), define
	\begin{equation*}\label{eq:J-definition}
		J_{n,\alpha,L}(t)
		:=
		\int_0^t
		\mathbf 1_{\{
			|Y_s^{(n)}|\le Ls^\alpha
			\}}\,ds.
	\end{equation*}
	
	We next prove the following tail probability estimate for  $J_{n,\alpha,L}(t)$.
	
	\begin{lemma}
		There exist constants \(c_H,C_H>0\), independent of \(n,\alpha,L,t\), such that, for every \(\lambda\ge1\),
		\begin{equation}\label{eq:J-tail}
			\mathbb P\left(
			J_{n,\alpha,L}(t)
			>
			C_HL\lambda\alpha^{-a}t^{\alpha+a}
			\right)
			\le
			C_H e^{-c_H\lambda}.
		\end{equation}
	\end{lemma}
	
	\begin{proof}
		
		Let \(p\ge1\) be an integer. By symmetry and
		\eqref{eq:multi-time-small-ball},
		\begin{align*}
			\mathbb E[J_{n,\alpha,L}(t)^p]
			&=
			p!
			\int_{0<s_1<\cdots<s_p<t}
			\mathbb P\left(
			|Y_{s_i}^{(n)}|
			\le Ls_i^\alpha,\
			i=1,\ldots,p
			\right)
			d\mathbf s
			\notag\\
			&\le
			p!(C_HL)^p
			\int_{0<s_1<\cdots<s_p<t}
			\prod_{i=1}^p
			s_i^\alpha
			(s_i-s_{i-1})^{-H}
			d\mathbf s,
			\label{eq:J-moment-start}
		\end{align*}
		where \(s_0=0\).
		
		With \(s_i=tu_i\),
		\begin{equation*}\label{eq:J-moment-scaled}
			\mathbb E[J_{n,\alpha,L}(t)^p]
			\le
			p!(C_HL)^p
			t^{p(\alpha+a)}
			I_{p,\alpha},
		\end{equation*}
		where
		\[
		I_{p,\alpha}
		:=
		\int_{0<u_1<\cdots<u_p<1}
		\prod_{i=1}^p
		u_i^\alpha
		(u_i-u_{i-1})^{-H}
		d\mathbf u,
		\qquad u_0=0.
		\]
		Since \(u_i\in(0,1)\),
		\[
		\prod_{i=1}^p u_i^\alpha\le u_1^\alpha.
		\]
		Hence
		\begin{align*}
			I_{p,\alpha}
			&\le
			\int_0^1
			u_1^{\alpha-H}
			\left[
			\int_{u_1<u_2<\cdots<u_p<1}
			\prod_{i=2}^p
			(u_i-u_{i-1})^{-H}
			du_2\cdots du_p
			\right]du_1.
		\end{align*}
		
		Here we use the standard Dirichlet (multivariate Beta) integral,
		which follows by repeated application of the Beta-function identity,
		\[
		\int_{\substack{x_1,\ldots,x_m>0\\ x_1+\cdots+x_m<1}}
		\prod_{j=1}^{m}x_j^{a_j-1}
		(1-x_1-\cdots-x_m)^{a_{m+1}-1}\,dx
		=
		\frac{\prod_{j=1}^{m+1}\Gamma(a_j)}
		{\Gamma\!\left(\sum_{j=1}^{m+1}a_j\right)},
		\qquad a_j>0.
		\]
		Indeed, with the increment variables
		\[
		x_1=u_1,\qquad x_i=u_i-u_{i-1}\quad (2\le i\le p),
		\]
		the corresponding parameters are
		\[
		\alpha+a,\quad
		\underbrace{a,\ldots,a}_{p-1\ \mathrm{times}},
		\quad 1,
		\]
		and hence
		\[
		I_{p,\alpha}
		\le
		\Gamma(a)^{p-1}
		\frac{\Gamma(\alpha+a)}
		{\Gamma(\alpha+pa+1)}.
		\]
		
		Next we use the standard Gamma-ratio inequality
		\begin{equation*}\label{eq:Gamma-ratio}
			\frac{\Gamma(x+A)}{\Gamma(x+B)}
			\le
			x^{-(B-A)},
			\qquad
			x>0,\quad A\ge0,\quad B-A\ge1.
		\end{equation*}
		With
		\[
		x=\alpha,\qquad
		A=a,\qquad
		B=pa+1,
		\]
		we have
		\[
		B-A=pa+1-a=pa+H\ge1,
		\]
		and therefore
		\begin{equation*}\label{eq:Gamma-applied}
			\frac{\Gamma(\alpha+a)}
			{\Gamma(\alpha+pa+1)}
			\le
			\alpha^{-pa-H}.
		\end{equation*}
		Taking \(p\)-th roots and using
		\((p!)^{1/p}\le p\), we obtain
		\begin{equation*}\label{eq:J-Lp-sharp}
			\|J_{n,\alpha,L}(t)\|_{L^p}
			\le
			C_HL\,
			p\,
			\alpha^{-a-H/p}
			t^{\alpha+a}.
		\end{equation*}
		Since \(\alpha\ge1\) by assumption,
		\begin{equation}\label{eq:J-Lp}
				\|J_{n,\alpha,L}(t)\|_{L^p}
				\le
				C_HL\,p\,\alpha^{-a}t^{\alpha+a}.
		\end{equation}
		
		Consequently, there exist constants \(c_H,C_H>0\) such that,
		for every \(\lambda\ge1\),
		\begin{equation*}
				\mathbb P\left(
				J_{n,\alpha,L}(t)
				>
				C_HL\lambda\alpha^{-a}t^{\alpha+a}
				\right)
				\le
				C_H e^{-c_H\lambda}.
		\end{equation*}
		Indeed, set
		\[
		S:=L\alpha^{-a}t^{\alpha+a}.
		\]
		The estimate \eqref{eq:J-Lp} then gives that, for every integer \(p\ge1\),
		\[
		\|J_{n,\alpha,L}(t)\|_{L_p}
		\le C_0 p S,
		\]
		where \(C_0>0\) depends only on \(H\). Equivalently,
		\[
		\mathbb E\left[J_{n,\alpha,L}(t)^p\right]
		\le (C_0pS)^p.
		\]
		
		Let \(C_1>0\), to be chosen below. By Markov's inequality, for every
		integer \(p\ge1\) and every \(\lambda\ge1\),
		\begin{align}
			\mathbb P\left(
			J_{n,\alpha,L}(t)>C_1\lambda S
			\right)
			&\le
			\frac{
				\mathbb E[J_{n,\alpha,L}(t)^p]
			}{
				(C_1\lambda S)^p
			}
			\notag\\
			&\le
			\left(
			\frac{C_0p}{C_1\lambda}
			\right)^p.
			\label{eq:markov-tail}
		\end{align}
		Notice in particular that \(L\), \(\alpha\), and \(t\) have cancelled
		from the right-hand side.
		
		We first consider sufficiently large \(\lambda\). Choose, for example,
		\[
		p=\lfloor\lambda\rfloor .
		\]
		If \(\lambda\ge2\), then
		\[
		\frac{\lambda}{2}\le p\le\lambda.
		\]
		Consequently, choosing \(C_1\ge 2C_0\), we obtain from
		\eqref{eq:markov-tail}
		\[
		\mathbb P\left(
		J_{n,\alpha,L}(t)>C_1\lambda S
		\right)
		\le
		\left(\frac12\right)^p
		\le
		\exp\left(
		-\frac{\log 2}{2}\lambda
		\right),
		\qquad \lambda\ge2.
		\]
		Thus, with
		\[
		c_H:=\frac{\log 2}{2},
		\]
		we have
		\[
		\mathbb P\left(
		J_{n,\alpha,L}(t)>C_1\lambda S
		\right)
		\le e^{-c_H\lambda},
		\qquad \lambda\ge2.
		\]
		
		It remains to consider the bounded range \(1\le\lambda<2\).
		Here we simply use the trivial estimate
		\[
		\mathbb P\left(
		J_{n,\alpha,L}(t)>C_1\lambda S
		\right)
		\le1.
		\]
		Since \(\lambda<2\),
		\[
		1
		\le
		e^{2c_H}e^{-c_H\lambda}.
		\]
		Hence, after setting
		\[
		C_H:=\max\{C_1,e^{2c_H}\},
		\]
		and increasing the constant in the threshold if necessary, we obtain
		for every \(\lambda\ge1\),
		\[
		\mathbb P\left(
		J_{n,\alpha,L}(t)
		>
		C_HL\lambda\alpha^{-a}t^{\alpha+a}
		\right)
		\le
		C_He^{-c_H\lambda}.
		\]
		The constants \(c_H,C_H>0\) depend only on \(H\) and are independent	of \(n,\alpha,L,t\), which proves \eqref{eq:J-tail}.
	\end{proof}
	
	Recall that equation \eqref{eq:J-tail} holds whenever $\alpha\geq 1$ and $L>0$. We now specify sequences \(\alpha_n\) and $L_n$ to which we will apply the latter equation.  First we define
		\begin{equation*}\label{eq:alpha-new}
		\alpha_n
		:=
		1+na-D_n,
		\qquad n\ge0,
	\end{equation*}
	where
	\[
	D_0:=0,
	\qquad
	D_n:=\sum_{k=1}^n\delta_k, \qquad \delta_n
	:=
	\frac{a}{(n+1)^2},
	\qquad n\ge1.
	\]
	Then
	\[
	\alpha_0=1
	\]
	and
	\begin{equation}\label{eq:alpha-recursion-new}
		\alpha_{n+1}
		=
		\alpha_n+a-\delta_{n+1}.
	\end{equation}	
	Since
	\[
	0<\delta_{n+1}<a,
	\]
	we have
	\[
	\alpha_{n+1}>\alpha_n.
	\]
	Consequently,
	\begin{equation*}\label{eq:alpha-at-least-one}
		\alpha_n\ge1\quad\text{for every }n\ge0.
	\end{equation*}
	Furthermore,
	\[
	D_\infty
	=
	a\sum_{k=1}^\infty\frac1{(k+1)^2}
	=
	a\left(\frac{\pi^2}{6}-1\right)
	<\infty,
	\]
	and
	\begin{equation}\label{eq:alpha-asymptotic}
		\alpha_n=an+O_H(1).
	\end{equation}
	Here \(O_H(1)\) denotes a quantity that is uniformly bounded in \(n\), with a bound that may depend on \(H\). In particular, there exist \(c_H>0\) and \(n_H<\infty\) such
	that
	\begin{equation*}
		\alpha_n\ge c_H n,
		\qquad n\ge n_H.
	\end{equation*}
	Next,  we shall define a sequence \((L_n)_{n\ge0}\), beginning with
	\[
	L_0:=1.
	\]
	For (given) $A>0$ we choose a deterministic constant
	\[
	C_*=C_*(H,T,A),
	\]
	independent of $n$, and recursively define
	\begin{equation}\label{eq:K-recursion-final}
		L_{n+1}
		:=
		C_*L_n
		\alpha_n^{-a}
		\log(n+2),
		\qquad n\ge0.
	\end{equation}
	Since \(L_0=1\), this defines the entire sequence \((L_n)_{n\ge0}\) deterministically.
	
	Observe that, with this choice of the sequences  $\alpha_n$ and $L_n$ the estimate \eqref{eq:J-tail} applies for every $n$, including \(n=0\). Later we will show that, roughly spoken, if
	\[
	R_n(t)\le L_n t^{\alpha_n},
	\]
	i.e. if
	\(R_n(t)\) is of order \(t^{\alpha_n}\), then \(R_{n+1}(t)\) is naturally of order \(t^{\alpha_n+a}\), where \(a=1-H\). Then we will iterate this inequality. Thus one expects to gain approximately \(a\) in the exponent at each
	iteration. 

	We now prove the following
	
	\begin{lemma}
		Let \(A>0\) (to be chosen later), and define
		\begin{equation*}\label{eq:lambda-nj}
			\lambda_{n,j}
			:=
			A\bigl[
			\log(n+2)+\log(j+2)
			\bigr], \qquad n,j \geq 0.
		\end{equation*}
		For each \(n\ge0\), set
		\[
		\gamma_n:=\alpha_n+a
		\]
		and
		\begin{equation*}\label{eq:time-mesh}
			q_n:=1+\frac1{\gamma_n},
			\qquad
			t_{n,j}:=Tq_n^{-j},
			\qquad j\ge0.
		\end{equation*}
		Define the events
		\begin{equation}\label{eq:bad-event}
			\begin{split}
				E_{n,j}
				:=
				\Bigl\{
				J_{n,\alpha_n,L_n}(t_{n,j})
				>
				C_HL_n
				\lambda_{n,j}
				\alpha_n^{-a}
				t_{n,j}^{\gamma_n}
				\Bigr\}.
			\end{split}
		\end{equation}
		Then for some positive and finite constant  $C_{H,T,A}$ it holds
		\begin{equation}\label{eq:uniform-J} 
			J_{n,\alpha_n,L_n}(t)
			\le
			C_{H,T,A}
			L_n
			\alpha_n^{-a}
			\log(n+2)
			t^{\alpha_{n+1}},
			\qquad 0<t\le T,
		\end{equation}
		provided none of the events \(E_{n,j}\), \(j\ge0\), defined in \eqref{eq:bad-event}, occurs. Here $\alpha_n$ and $L_n$ are the sequences chosen above.
	\end{lemma}
	
	\begin{proof}
		By \eqref{eq:alpha-recursion-new},
		\begin{equation}\label{eq:gamma-relation}
			\gamma_n
			=
			\alpha_{n+1}+\delta_{n+1}.
		\end{equation}
		Moreover, by the definition of the events $E_{n,j}$ and the constants $\lambda_{n,j}$, applying \eqref{eq:J-tail}, we obtain
		\begin{equation}\label{eq:bad-probability}
			\mathbb P(E_{n,j})
			\le
			C_H
			(n+2)^{-c_HA}
			(j+2)^{-c_HA}.
		\end{equation}

	On \(E_{n,j}^c\), using
	\eqref{eq:gamma-relation},
	\begin{align}
		J_{n,\alpha_n,L_n}(t_{n,j})
		&\le
		C_HA L_n
		\alpha_n^{-a}
		[\log(n+2)+\log(j+2)]
		\notag\\
		&\qquad\times
		t_{n,j}^{\alpha_{n+1}}
		t_{n,j}^{\delta_{n+1}}.
		\label{eq:grid-bound-1}
	\end{align}
	Now
	\begin{equation}\label{eq:delta-factor}
		t_{n,j}^{\delta_{n+1}}
		=
		T^{\delta_{n+1}}
		q_n^{-j\delta_{n+1}}.
	\end{equation}
	
	Since
	\[
	\gamma_n\asymp_H n+1,
	\qquad
	\delta_{n+1}\asymp_H(n+1)^{-2},
	\]
	and
	\[
	\log q_n
	=
	\log\left(1+\frac1{\gamma_n}\right)
	\asymp_H(n+1)^{-1},
	\]
	there exists \(c_H>0\) such that
	\begin{equation}\label{eq:q-decay}
		q_n^{-j\delta_{n+1}}
		\le
		\exp\left(
		-c_H\frac{j}{(n+1)^3}
		\right).
	\end{equation}
	Here, the notation $f_n \asymp_H g_n$ means that there exist constants $c_1, c_2$, only depending on  $H$, such that
	$$ c_1 f_n \leq g_n \leq c_2 g_n$$ 
	for every $n$. Consequently,
	\begin{equation}\label{eq:j-sup}
		\begin{split}
			&
			\sup_{j\ge0}
			[\log(n+2)+\log(j+2)]
			q_n^{-j\delta_{n+1}}
			\\
			&\hspace{3cm}
			\le
			C_H\log(n+2).
		\end{split}
	\end{equation}
	Indeed,
	\[
	\sup_{x\ge0}
	\log(x+2)
	\exp\left(-c_H\frac{x}{(n+1)^3}\right)
	\le
	C_H\log(n+2).
	\]
	
	Moreover,
	\[
	\sup_{n\ge0}T^{\delta_{n+1}}
	\le
	\max\{1,T^{a/4}\}
	<\infty.
	\]
	Combining
	\eqref{eq:grid-bound-1}--\eqref{eq:j-sup}, we obtain
	\begin{equation*}
		J_{n,\alpha_n,L_n}(t_{n,j})
		\le
		C_{H,T,A}
		L_n
		\alpha_n^{-a}
		\log(n+2)
		t_{n,j}^{\alpha_{n+1}}
	\end{equation*}
	whenever \(E_{n,j}\) does not occur.
	
	Now let
	\[
	t_{n,j+1}<t\le t_{n,j}.
	\]
	Since \(J_{n,\alpha_n,L_n}\) is nondecreasing,
	\[
	J_{n,\alpha_n,L_n}(t)
	\le
	J_{n,\alpha_n,L_n}(t_{n,j}).
	\]
	Furthermore,
	\[
	t_{n,j}\le q_nt
	\]
	and
	\[
	q_n^{\alpha_{n+1}}
	=
	\left(1+\frac1{\gamma_n}\right)^{\alpha_{n+1}}
	\le e.
	\]
	Hence, after enlarging the deterministic constant,
	\begin{equation*}
			J_{n,\alpha_n,L_n}(t)
			\le
			C_{H,T,A}
			L_n
			\alpha_n^{-a}
			\log(n+2)
			t^{\alpha_{n+1}},
			\qquad 0<t\le T,
	\end{equation*}
	provided none of the events \(E_{n,j}\), \(j\ge0\), occurs, which proves \eqref{eq:uniform-J}.
	\end{proof}
	
	We now combine the (uniform) estimates obtained so far with the recursion for \(R_n\) to prove, simultaneously for all \(n\), the inductive bounds
	\[
	R_n(t)\le L_n t^{\alpha_n},
	\qquad t\in[0,T].
	\]
	
	\begin{lemma} \label{Rnrecursion}
		Let $\alpha_n$ be defined as in \eqref{eq:alpha-recursion-new} and let $L_n$ be defined as in \eqref{eq:K-recursion-final}, where the constant \[
		C_*=C_*(H,T,A)
		\] in the definition of $L_n$ is chosen large enough such that \eqref{eq:uniform-J} with this constant holds. Then for every $\varepsilon>0$ there exists $A>0$ large enough such that, with probability at least $1-\varepsilon$, it holds
			\[
		R_n(t)\le L_n t^{\alpha_n},
		\qquad t\in[0,T], \, \quad n \in \mathbb{N}.
		\]
	\end{lemma}
	
	\begin{proof}
		We give a proof using induction over $n$. Recall the definition of the events $E_{n,j}$ in \eqref{eq:bad-event}. In particular, note that all the events \(E_{n,j}\) are deterministic functions of
		\(H,T,A,n,j\) and the driving Gaussian Volterra process; their
		definition does not depend on whether any previous induction
		step has succeeded. By \eqref{eq:bad-probability},
	\[
	\sum_{n=0}^\infty\sum_{j=0}^\infty
	\mathbb P(E_{n,j})
	\le
	C_H
	\left(
	\sum_{n=0}^\infty(n+2)^{-c_HA}
	\right)
	\left(
	\sum_{j=0}^\infty(j+2)^{-c_HA}
	\right).
	\]
	The right-hand side tends to zero as \(A\to\infty\). Fix \(\varepsilon>0\).
	Choose \(A=A(\varepsilon,H)\) sufficiently large that
	\begin{equation*}\label{eq:failure-sum}
		\sum_{n=0}^\infty\sum_{j=0}^\infty
		\mathbb P(E_{n,j})
		<\varepsilon.
	\end{equation*}
	Define
	\begin{equation*}\label{eq:good-event}
		G_\varepsilon
		:=
		\bigcap_{n=0}^\infty
		\bigcap_{j=0}^\infty
		E_{n,j}^c.
	\end{equation*}
	By the union bound,
	\begin{equation*}\label{eq:good-event-probability}
		\mathbb P(G_\varepsilon)\ge1-\varepsilon.
	\end{equation*}
	
	We claim that on \(G_\varepsilon\),
	\begin{equation}\label{eq:induction-claim}
		R_n(t)
		\le
		L_nt^{\alpha_n},
		\qquad
		n\ge0,\quad 0\le t\le T.
	\end{equation}
	
	For \(n=0\),
	\[
	R_0(t)=t=L_0t^{\alpha_0}.
	\]
	Suppose \eqref{eq:induction-claim} holds at level \(n\). Then,
	by \eqref{eq:R-recursion},
	\begin{align*}
		R_{n+1}(t)
		&\le
		\int_0^t
		\mathbf 1_{\{
			|Y_s^{(n)}|\le R_n(s)
			\}}\,ds
		\notag\\
		&\le
		\int_0^t
		\mathbf 1_{\{
			|Y_s^{(n)}|\le L_ns^{\alpha_n}
			\}}\,ds
		\notag\\
		&=
		J_{n,\alpha_n,L_n}(t).
		\label{eq:induction-J}
	\end{align*}
	Since \(G_\varepsilon\subseteq E_{n,j}^c\) for every \(j\),
	\eqref{eq:uniform-J} and \eqref{eq:K-recursion-final} imply
	\[
	R_{n+1}(t)
	\le
	L_{n+1}t^{\alpha_{n+1}}.
	\]
	This proves \eqref{eq:induction-claim} for every \(n\).
\end{proof}

	We finally show that the bounds obtained in Lemma \ref{Rnrecursion} force the uncertainty radius \(R_n(t)\) to vanish as \(n\to\infty\),
	so that the lower and upper extremal iterations collapse onto the same limit, yielding pathwise uniqueness.
	
	\begin{proof}[Proof of Theorem \ref{thm:sign-pathwise-uniqueness}]
		From \eqref{eq:K-recursion-final},
	\begin{equation*}\label{eq:K-product}
		L_n
		=
		C_*^n
		\prod_{k=0}^{n-1}
		\alpha_k^{-a}\log(k+2).
	\end{equation*}
	By \eqref{eq:alpha-asymptotic}, there exist
	\(c_H,C_H>0\) such that, for all sufficiently large \(k\),
	\[
	c_H(k+1)\le\alpha_k\le C_H(k+1).
	\]
	Consequently,
	\[
	L_n
	\le
	C^n
	\frac{
		\displaystyle\prod_{k=0}^{n-1}\log(k+2)
	}{
		(n!)^a
	}
	\]
	for some deterministic \(C=C(H,T,A)<\infty\).
	Using
	\[
	\prod_{k=0}^{n-1}\log(k+2)
	\le
	[\log(n+1)]^n
	\]
	and Stirling's formula,
	\[
	(n!)^{1/n}\sim\frac ne,
	\]
	we obtain
	\begin{equation*}\label{eq:K-root}
		L_n^{1/n}
		\le
		C
		\frac{\log(n+1)}
		{(n!)^{a/n}}
		=
		O\left(
		\frac{\log n}{n^a}
		\right).
	\end{equation*}
	Since \(a=1-H>0\),
	\begin{equation}\label{eq:K-root-zero}
			L_n^{1/n}\longrightarrow0.
	\end{equation}
	
	Now fix \(t\in(0,T]\). Let $\varepsilon>0$. By Lemma \ref{Rnrecursion} we can choose $A>0$ large enough in the definition of $L_n$ such that with probability at least $1-\varepsilon$ it holds
	\[
	R_n(t)
	\le
	L_nt^{\alpha_n}
	=
	L_nt^{1+na-D_n}.
	\]
	Since
	\[
	0\le D_n\le D_\infty<\infty,
	\]
	there exists a finite deterministic constant \(C_t\) such that
	\[
	t^{1-D_n}\le C_t
	\qquad\text{for every }n.
	\]
	Therefore
	\begin{align*}
		R_n(t)
		&\leq
		t^{1-D_n}
		L_n(t^a)^n
		\notag\\
		&\le
		C_t
		\left[
		L_n^{1/n}t^a
		\right]^n.
		\label{eq:R-final-decay}
	\end{align*}
	By \eqref{eq:K-root-zero},
	\[
	L_n^{1/n}t^a\longrightarrow0,
	\]
	and hence
	\begin{equation}\label{eq:R-zero}
			R_n(t)\longrightarrow0
			\qquad\text{for every }t\in(0,T].
	\end{equation}
	
	Let \(A^1,A^2\) be two fixed points of \(\Phi\). By
	\eqref{eq:fixed-point-squeeze},
	\[
	L_n(t)\le A_t^i\le U_n(t),
	\qquad i=1,2.
	\]
	Thus
	\[
	|A_t^1-A_t^2|
	\le
	U_n(t)-L_n(t)
	=
	2R_n(t).
	\]
	By \eqref{eq:R-zero}, for every fixed \(t\in[0,T]\),
    \[
    \mathbb P(A_t^1=A_t^2)\ge 1-\varepsilon.
    \]
    Since \(\varepsilon>0\) was arbitrary, letting
    \(\varepsilon\downarrow0\) yields
    \[
    \mathbb P(A_t^1=A_t^2)=1
    \]
    for every fixed \(t\in[0,T]\).
    
    To obtain simultaneous equality in time, we first consider
    all rational \(t\in[0,T]\). Since the rationals are countable,
    \[
    \mathbb P\left(
    A_t^1=A_t^2
    \text{ for all }t\in\mathbb Q\cap[0,T]
    \right)=1.
    \]
    Finally, by continuity of \(A^1\) and \(A^2\),
    \[
    \mathbb P\left(
    A^1=A^2\text{ on }[0,T]
    \right)=1.
    \]
	Equivalently,
	\[
	\mathbb P\left(
	X_t^1=X_t^2
	\text{ for all }t\in[0,T]
	\right)=1.
	\]
	
	This proves pathwise uniqueness.
\end{proof}

\bigskip
\noindent\textbf{Declaration of AI use.}
ChatGPT (OpenAI) was used to assist with mathematical discussions,
drafting, and language editing. The authors take full responsibility
for the content of the manuscript.

\bibliographystyle{plain}
\bibliography{references}

\end{document}